\documentclass[10pt, a4paper]{amsart}

\usepackage[latin1]{inputenc}
\usepackage{amsmath, amssymb, amsfonts, graphicx}
\usepackage{color}

\newcommand{\tp}{{\sf T}}
\newcommand{\nulltype}{\varnothing}
\newcommand{\downop}[1]{[\kern-1.65pt[#1]\kern-1.65pt]}

\newcommand{\R}{\mathbb{R}}

\newcommand{\Acal}{\mathcal{A}}
\newcommand{\Ccal}{\mathcal{C}}
\newcommand{\Fcal}{\mathcal{F}}
\newcommand{\Gcal}{\mathcal{G}}
\newcommand{\Hcal}{\mathcal{H}}
\newcommand{\Kcal}{\mathcal{K}}

\newcommand{\Scal}{\mathcal{S}}
\newcommand{\Tcal}{\mathcal{T}}

\newtheorem{theorem}{Theorem}
\newtheorem{lemma}[theorem]{Lemma}

\DeclareMathOperator{\tr}{tr}
\DeclareMathOperator{\Hom}{Hom}
\DeclareMathOperator{\exparam}{ex}

\DeclareMathOperator{\im}{Im}

\newcommand{\nonedge}{\raise1pt\hbox{\includegraphics{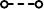}}}
\newcommand{\edge}{\raise1pt\hbox{\includegraphics{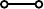}}}
\newcommand{\emptythree}{\lower3pt\hbox{\includegraphics{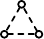}}}
\newcommand{\oneedge}{\lower3pt\hbox{\includegraphics{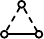}}}
\newcommand{\twoclaw}{\lower3pt\hbox{\includegraphics{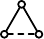}}}
\newcommand{\triangleflag}{\lower3pt\hbox{\includegraphics{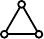}}}

\newcommand{\bnonedge}{\raise1pt\hbox{\includegraphics{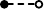}}}
\newcommand{\bedge}{\raise1pt\hbox{\includegraphics{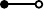}}}
\newcommand{\bempty}{\lower3pt\hbox{\includegraphics{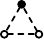}}}
\newcommand{\bonezero}{\lower3pt\hbox{\includegraphics{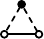}}}
\newcommand{\boneone}{\lower3pt\hbox{\includegraphics{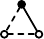}}}
\newcommand{\btwoclaw}{\lower3pt\hbox{\includegraphics{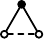}}}
\newcommand{\btwoone}{\lower3pt\hbox{\includegraphics{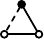}}}

\def\cdotwc{{\,\cdot\,}}
\def\forget#1{{\downarrow} #1}  
\def\expec{\mathop{\mathbb E}}

\newcommand{\textred}[1]{#1}

\title{Flag Algebras: A First Glance}

\author{Marcel K. de Carli Silva}

\author{Fernando Mário de Oliveira Filho} 

\address{M.K. de Carli Silva and F.M. de Oliveira Filho, Instituto de
  Matemática e Estatística, Universidade de São Paulo, Rua do Matão,
  1010, 05508-090 São Paulo/SP, Brazil.}
 \email{\{mksilva, fmario\}@ime.usp.br}

\author{Cristiane Maria Sato}
\address{C.M. Sato, Centro de Matemática, Computação e Cognição,
  Universidade Federal do ABC, Rua Abolição, S/N, 09210-180 Santo
  André/SP, Brazil.}
\email{c.sato@ufabc.edu.br}

\date{\today}

\begin{document}

\begin{abstract} 
  The theory of flag algebras, introduced by Razborov in~2007, has
  opened the way to a systematic approach to the development of
  computer-assisted proofs in extremal combinatorics. It makes it
  possible to derive bounds for parameters in extremal combinatorics
  with the help of a computer, in a semi-automated manner.  This
  article describes the main points of the theory in a complete way,
  using Mantel's theorem as a guiding example.
\end{abstract}

\maketitle

\markboth{M.K. de Carli Silva, F.M. de Oliveira Filho, and C.M.
  Sato}{Flag Algebras: A First Glance}


\section{Introduction}
\label{sec:intro}

Mantel's theorem, perhaps the first result in extremal graph theory,
was motivated by a problem proposed by W.~Mantel in an issue of the
journal \textit{Wiskundige Opgaven}, published by the
KWG~\cite{Mantel1910}:

\begin{quote}
\begin{center}
Vraagstuk XXVIII
\end{center}

\textbf{K 13 a.} Er zijn eenige punten gegeven waarvan geen vier in
een zelfde vlak liggen. Hoeveel rechten kan men hoogstens tusschen die
punten trekken zonder driehoeken te
vormen?\quad(W.~Mantel.)\footnote{Problem~XXVIII: Given are some
  points, no four of which lie on the same plane. How many lines at
  most can one draw between the points without forming triangles?}
\end{quote}

In the language of graph theory, Mantel's problem asks for the maximum
number of edges that a graph without triangles can have: the
restriction that no four points lie on the same plane is there exactly
to ensure that only triangles between the given points can be formed
when lines are drawn. 

A triangle-free graph on~$n$ vertices can be constructed as follows:
divide the vertex set into two parts of~$\lfloor n/2 \rfloor$
and~$\lceil n/2\rceil$ vertices each and add all edges between the
parts. The resulting graph is bipartite, and hence in particular
triangle-free, and has~$\lfloor n^2 / 4 \rfloor$ edges. Mantel's
theorem states that this is an \textit{extremal example}, the best one
can do: \textit{every triangle-free graph on~$n$ vertices has at
  most~$\lfloor n^2 / 4 \rfloor$ edges}.

This answer to Mantel's problem appeared in the same issue of
\textit{Wiskundige Opgaven}. There it is mentioned that solutions were
provided by Mantel and several others; a proof by W.A. Wythoff
(1865--1939), a former student of D.J.~Korteweg (1848--1941), is
included.

The theory of flag algebras allows us to computationally tackle
extremal graph theory problems such as Mantel's problem and to obtain
results such as Mantel's theorem. To understand how this is done, we
first need to define exactly which extremal problems we consider.

The \textit{size} of a graph~$G$ is its number of vertices~$|V(G)|$
and is denoted by~$|G|$. For~$U \subseteq V(G)$, we denote by~$G[U]$
the subgraph of~$G$ \textit{induced} by~$U$, that is, the subgraph
of~$G$ with vertex set~$U$ and all the edges of~$G$ between vertices
of~$U$.  For graphs~$F$ and~$G$, let~$p(F; G)$ be the probability that
a set~$U \subseteq V(G)$ with~$|U| = |F|$, chosen uniformly at random,
is such that~$G[U]$ is isomorphic to~$F$. We say that~$p(F; G)$ is the
\textit{density} of~$F$ in~$G$. In other words, if~$c(F; G)$ is the
number of times~$F$ occurs as an induced subgraph of~$G$, then
\[
p(F; G) = c(F; G) \binom{|G|}{|F|}^{-1}.
\]

Let~$\Hcal$ be a collection of graphs. A graph~$G$ is
\textit{$\Hcal$-free} if no induced subgraph of~$G$ is isomorphic to a
graph in~$\Hcal$. A fundamental problem in extremal graph theory is to
determine, for a given graph~$C$, the maximum asymptotic density
of~$C$ in $\Hcal$-free graphs
\begin{equation}
\label{eq:exparam}
\exparam(C, \Hcal) = \sup_{(G_k)_{k \geq 0}} \limsup_{k\to\infty} p(C;
G_k),
\end{equation}
where the supremum is taken over all sequences~$(G_k)_{k \geq 0}$ of
$\Hcal$-free graphs that are \textit{increasing}, i.e., with
$(|G_k|)_{k \geq 0}$ is strictly increasing.

Mantel's theorem shows
that~$\exparam(\edge, \{\triangleflag\}) \leq 1/2$. Together with the
extremal example described above, we actually
have~$\exparam(\edge, \{\triangleflag\}) = 1/2$.

Let~$\Gcal$ be the set of all finite $\Hcal$-free graphs taken up to
isomorphism. An increasing sequence~$(G_k)_{k \geq 0}$ is
\textit{convergent} if~$\lim_{k\to\infty} p(F; G_k)$ exists for
every~$F \in \Gcal$. Every increasing sequence of $\Hcal$-free graphs
has a convergent subsequence. Indeed, densities are numbers
in~$[0,1]$, so for~$k \geq 0$ the function~$F \mapsto p(F; G_k)$ can
be identified with a point in~$[0,1]^\Gcal$, which is a compact space
by Tychonoff's theorem.

In~\eqref{eq:exparam} we may therefore restrict ourselves to
convergent sequences and this allows us to work with their
limits. Call~$\phi\colon \Gcal \to \R$ a \textit{limit functional} if
there is a convergent sequence~$(G_k)_{k \geq 0}$ of $\Hcal$-free
graphs such that
\[
\phi(F) = \lim_{k \to \infty} p(F; G_k)
\]
for all~$F \in \Gcal$ and let~$\Phi$ denote the set of all limit
functionals.  \textit{Then computing~$\exparam(C, \Hcal)$ is the same as
solving an optimization problem over~$\Phi$:}
\begin{equation}
\label{eq:exparam-opt}
\exparam(C, \Hcal) = \sup\{\,\phi(C) : \phi \in \Phi\,\}.
\end{equation}

This is just a rewording of the original problem, but it emphasizes
that the difficulty here lies in understanding~$\Phi$. This set may be
very complex and computationally intractable, but to get an upper
bound for~$\exparam(C, \Hcal)$ we do not need to work with~$\Phi$.
Instead, we may look for a nice \textit{relaxation} of~$\Phi$, that
is, a set~$\Phi' \supseteq \Phi$ for which we can solve the
optimization problem. A first and obvious relaxation would be to
take~$\Phi' = [0,1]^\Gcal$. Solving the optimization problem is then
trivial, but we always get the bound~$\exparam(C, \Hcal) \leq 1$. The
difficulty lies in managing the trade-off between the quality of the
relaxation and its tractability.

The theory of flag algebras~\cite{Razborov2007}, developed by the
Russian mathematician Alexander Razborov, winner of the Nevanlinna
Prize in~1990 and the Gödel Prize in~2007, gives us
computationally-tractable relaxations of~$\Phi$ that have displayed
good quality in practice. We may then use the computer to solve the
corresponding optimization problems, thus obtaining upper bounds
for~$\exparam(C, \Hcal)$ that are often tight. Perhaps the most
attractive feature in the theory is that the whole process is
more-or-less automatic: obtaining the relaxation and solving the
corresponding problems is basically a computational matter. So the
theory of flag algebras allows us to harness computational power and
apply it to problems in extremal combinatorics; it can be understood
as part of the growing trend for the use of computers in mathematics.

Razborov credits Bondy~\cite{Bondy1997} with a predecessor of the
theory of flag algebras. Bondy applies counting techniques to the
Caccetta-Häggkvist conjecture\footnote{The Cacceta-Häggkvist
  conjecture states that every simple directed graph on~$n$ vertices
  with outdegree at least~$r$ has a cycle with length at
  most~$\lceil n / r \rceil$.}  and illustrates his idea on Mantel's
theorem. Here is a proof
that~$\exparam(\edge, \{\triangleflag\}) \leq 1/2$ that is a rewording
of the proof by Bondy in terms of densities and limit
functionals. This proof is a first glance into the theory of flag
algebras; in it we will derive by hand some constraints on limit
functionals of sequences of triangle-free graphs and then give an
explicit simple relaxation of~$\Phi$ from which Mantel's theorem will
follow.

A triangle-free graph may have three different graphs on three
vertices as induced subgraphs: the empty graph~$\emptythree$, the
graph with one edge~$\oneedge$, and the graph with two
edges~$\twoclaw$. (Nonedges are represented by dashed lines.) Let~$G$
be a triangle-free graph. Every edge of~$G$ belongs to~$|G| - 2$
induced subgraphs with three vertices, whence
\[
p(\oneedge; G) + 2 p(\twoclaw; G) = 3 p(\edge; G).
\]
This is valid for every triangle-free graph~$G$, hence also for a
limit functional~$\phi$:
\[
\phi(\oneedge) + 2 \phi(\twoclaw) = 3\phi(\edge).
\]
We have our first constraint satisfied for all~$\phi \in \Phi$.

A second constraint comes from the identity
\[
p(\twoclaw; G) = \binom{|G|}{3}^{-1} \sum_{v \in V(G)} \binom{d(v)}{2},
\]
where~$d(v)$ is the degree of vertex~$v$. To rewrite the right-hand
side above, we need to extend the definition of the density
function~$p$ to partially-labeled graphs. Say~$F$ and~$G$ are graphs each
having a special vertex labeled~$1$, and let~$x_1$ be the vertex
of~$G$ labeled~$1$. Let~$p(F; G)$ be the probability that a set~$U
\subseteq V(G) \setminus \{ x_1 \}$ with~$|U| = |F| - 1$, chosen
uniformly at random, is such that~$G[U \cup \{ x_1 \}]$ is isomorphic
to~$F$ via a label-preserving isomorphism, that is, an isomorphism
that takes the labeled vertex of~$F$ to the labeled vertex of~$G$.

For~$v \in V(G)$, denote by~$G^v$ the labeled graph obtained from~$G$ by
labeling vertex~$v$ with label~$1$. Let~$\btwoclaw$ denote the
labeled graph obtained from~$\twoclaw$ by labeling the vertex of
degree two with label~$1$; similarly for other graphs the solid
vertex will be the labeled vertex. Then for a triangle-free graph~$G$
we have
\begin{equation}
\label{eq:intro_label}
\begin{split}
p(\twoclaw; G)
&=\binom{|G|}{3}^{-1} \sum_{v \in V(G)} \binom{d(v)}{2}
\\
&= \binom{|G|}{3}^{-1} \sum_{v \in V(G)} p(\btwoclaw; G^v) \binom{|G| -
  1}{2}\\
&= \frac{3}{|G|} \sum_{v \in V(G)} p(\btwoclaw; G^v).
\end{split}
\end{equation}

Now comes a key observation. As the size of~$G$ goes to infinity,
$p(\btwoclaw; G^v)$ goes
to~$p(\bedge; G^v)^2$. This is not hard to prove (do it!), but the
intuition should be clear: if~$G$ is very large, then choosing a subset
of~$V(G) \setminus \{v\}$ of size~$2$ uniformly at random is basically
the same as choosing two vertices in~$V(G) \setminus \{v\}$ independently
--- the probability of choosing the same vertex twice becomes
negligible as~$|G|$ grows larger.

So let~$\phi$ be the limit functional of a convergent
sequence~$(G_k)_{k \geq 0}$ of triangle-free graphs. Then
\begin{equation}
\label{eq:phi-first-fact}
\begin{split}
\phi(\twoclaw) = \lim_{k \to \infty} p(\twoclaw; G_k)
&= \lim_{k\to\infty} \frac{3}{|G_k|} \sum_{v \in V(G_k)} p(\btwoclaw;
G_k^v)\\
&= \lim_{k\to\infty} \frac{3}{|G_k|} \sum_{v \in V(G_k)} p(\bedge;
G_k^v)^2.
\end{split}
\end{equation}
Now, for any triangle-free graph~$G$ the Cauchy-Schwarz inequality
gives
\[
\sum_{v \in V(G)} p(\bedge; G^v)^2 \geq \frac{1}{|G|} \biggl(\sum_{v
  \in V(G)} p(\bedge; G^v)\biggr)^2.
\]
Together with~\eqref{eq:phi-first-fact} and
\[
\sum_{v \in V(G)} p(\bedge; G^v) (|G| - 1) = 2 p(\edge; G)
\binom{|G|}{2}
\]
we get
\[
\phi(\twoclaw) \geq \lim_{k\to\infty} 3 p(\edge; G_k)^2 =
3\phi(\edge)^2.
\]

So every limit functional~$\phi$ satisfies the constraints
\begin{align*}
\phi(\oneedge) + 2 \phi(\twoclaw) &= 3\phi(\edge),\\
\phi(\twoclaw)&\geq 3\phi(\edge)^2.
\end{align*}
What do we get in~\eqref{eq:exparam-opt} if we optimize over the set~$\Phi'$ of
all~$\phi\colon \Gcal \to [0, 1]$ satisfying the constraints above?
Well, suppose~$\phi \in \Phi'$. Multiply the second
constraint by~$2$ and subtract it from the first to get
\[
\phi(\oneedge) \leq 3\phi(\edge) - 6\phi(\edge)^2.
\]
Since~$\phi(\oneedge) \geq 0$, we then have~$\phi(\edge) \leq 1/2$. So
the optimal value of~\eqref{eq:exparam-opt} with~$\Phi'$ instead
of~$\Phi$ is at most~$1/2$,
hence~$\exparam(\edge, \{\triangleflag\}) \leq 1/2$.

In the following sections the main points of Razborov's theory of flag
algebras are developed. Unless otherwise noted, every definition and
result presented here can be found in Razborov's original
paper~\cite{Razborov2007}.


\section{Types and flags}

In the introduction, we derived valid inequalities for~$\Phi$ by
combining densities of partially-labeled graphs as
in~\eqref{eq:intro_label}. In the next few sections we will develop
Razborov's theory of flag algebras, which automates this process.  The
discussion will be focused on families of graphs for concreteness,
though one of the most attractive features of the theory is that it
applies to a whole range of structures, including directed graphs,
hypergraphs, and permutations.

For an integer $k \geq 0$, write $[k] = \{1,\ldots,k\}$.  Fix a
family~$\Hcal$ of forbidden subgraphs. A {\it type\/} of {\it size
  $k$\/} is an $\Hcal$-free graph~$\sigma$ with~$V(\sigma) = [k]$. We
can think of it as a graph with vertices labeled with~$1,\dotsc,k$,
whereas we regard graphs as unlabeled. The empty type is denoted
by~$\nulltype$.

Let~$\sigma$ be a type of size~$k$ and~$F$ be a graph on at least~$k$
vertices. An {\it embedding\/} of~$\sigma$ into~$F$ is an injective
function~$\theta\colon [k] \to V(F)$ that defines an isomorphism
between~$\sigma$ and the subgraph of~$F$ induced by~$\im\theta$.

A {\it $\sigma$-flag\/} is a pair~$(F, \theta)$ where~$F$ is an
$\Hcal$-free graph and~$\theta$ is an embedding of~$\sigma$ into~$F$.
So a $\sigma$-flag is a partially-labeled graph that avoids~$\Hcal$
and whose labeled part is a copy of~$\sigma$. When the embedding
itself is not important, we will drop it, speaking simply
of the $\sigma$-flag~$F$.

The {\it labeled vertices\/} of~$(F, \theta)$ are the vertices in the
image of~$\theta$. Note that an $\nulltype$-flag is just \textred{an
  $\Hcal$-free graph}.  Any type~$\sigma$ of size~$k$ can also be seen
as the $\sigma$-flag~$(\sigma, \theta)$ where $\theta$ is the identity
on~$[k]$.

Isomorphism between $\sigma$-flags is defined just as for graphs, but
now the labels should also be preserved by the bijection. More
precisely, $\sigma$-flags~$(F, \theta)$ and~$(G, \eta)$ are {\it
  isomorphic\/} if there is a graph
isomorphism~$\rho\colon V(F) \to V(G)$ between~$F$ and~$G$ such
that~$\rho(\theta(i)) = \eta(i)$ for~$i = 1$,
\dots,~$|\sigma|$. Write~$(F, \theta) \simeq (G, \eta)$
when~$(F, \theta)$ and~$(G, \eta)$ are isomorphic, or
simply~$F \simeq G$ when the embeddings are not important. In the
introduction, this notion was used only for $\sigma$-flags where
$\sigma$ is the type of size~$1$. Figure~\ref{fig:someflags} shows
some flags of different types.

\begin{figure}
\def\flag#1{\includegraphics{flags-#1.pdf}}
\begin{center}
\begin{tabular}{ccccc}
\flag{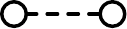}&\flag{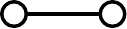}&\flag{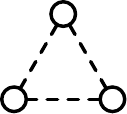}&\flag{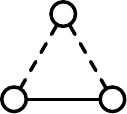}&\flag{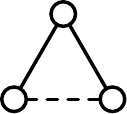}\\[5mm]
&\flag{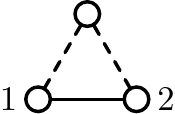}&\flag{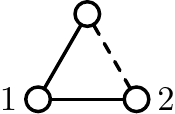}&\flag{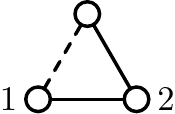}
\end{tabular}
\end{center}

\caption{Let~$\Hcal = \{\protect\triangleflag\}$. On the top row we
  have all $\nulltype$-flags of sizes~$2$ and~$3$, up to isomorphism
  (nonedges are shown as dashed lines);
  notice that the triangle itself is not a flag. On the bottom row we
  have all flags of type~$\sigma =
  \raise1pt\hbox{$\scriptscriptstyle 1$} \protect\edge
  \raise1pt\hbox{$\scriptscriptstyle 2$}$; notice that the last two of
  these flags are not isomorphic, since the isomorphism has to
  preserve the labels.}
\label{fig:someflags}
\end{figure}

For~$n \geq |\sigma|$, denote by~$\Fcal_n^\sigma$ the set of all
$\sigma$-flags of size~$n$, taken up to isomorphism; denote
by~$\Fcal^\sigma$ the set of all $\sigma$-flags taken up to
isomorphism. Note that the set~$\Gcal$ of all $\Hcal$-free graphs is
simply~$\Fcal^\nulltype$. A type~$\sigma$ is \textit{degenerate}
if~$\Fcal^\sigma$ is finite. If~$\sigma$ is nondegenerate,
then~$\Fcal_n^\sigma \neq \emptyset$ for all~$n \geq |\sigma|$. It is
easy to construct a family~$\Hcal$ for which there are degenerate
types: take for instance~$\Hcal$ as the set of all graphs with~1000
vertices containing at least one triangle. Then the triangle itself
is~$\Hcal$-free, and hence is a type, but there are no
$\triangleflag$-flags of size~$\geq 1000$.

\textit{From now on, we assume that all types are nondegenerate.} In
particular, every time a result about $\sigma$-flags is stated, it is
implicitly assumed that~$\sigma$ is nondegenerate.


\section{Density}

The definition of density given in the introduction can be extended to
$\sigma$-flags as follows.  We say that $\sigma$-flags~$F_1$,
\dots,~$F_t$ {\it fit\/} in a $\sigma$-flag~$G$ if
\[
  |G|-|\sigma| \geq (|F_1|-|\sigma|) + \cdots + (|F_t|-|\sigma|).
\]
Let~$F_1$, \dots,~$F_t$ and~$(G, \theta)$ be $\sigma$-flags such
that~$F_1$, \dots,~$F_t$ fit in~$G$. Consider the following
experiment: \textred{choose} pairwise-disjoint sets~$U_1$,
\dots,~$U_t \subseteq V(G) \setminus \im\theta$ of unlabeled vertices
of~$G$ with~$|U_i| = |F_i|-|\sigma|$ uniformly at random.
Let~$p(F_1, \ldots, F_t; G)$ be the probability that the $\sigma$-flag
$(G[U_i \cup \im\theta], \theta)$ is isomorphic to~$F_i$ for~$i = 1$,
\dots,~$t$. This is the {\it density\/} of~$F_1$, \dots,~$F_t$
in~$G$. For $\nulltype$-flags and~$t=1$, this definition coincides
with the usual notion of density for graphs. In the introduction we
also extended the definition of density to graphs with one labeled
vertex; this corresponds to taking~$t = 1$ and the only type of size~1
as~$\sigma$.

Say~$|F| \leq n \leq |G|$. To embed~$F$ into~$G$, we may first try to
embed~$F$ into a $\sigma$-flag~$F'$ of size~$n$ and then embed~$F'$
into~$G$. This gives us another way to compute~$p(F; G)$:
\begin{equation}
p(F; G) = \sum_{F' \in \Fcal_n^\sigma} p(F; F') p(F'; G).
\label{eq:chain-rule}
\end{equation}
This identity can be generalized, giving us the \textit{chain rule}:

\begin{theorem}
If $F_1$, \dots,~$F_t$, and~$G$ are $\sigma$-flags such that~$F_1$,
\dots,~$F_t$ fit in~$G$, then for every~$1 \leq s \leq t$ and
every~$n$ such that~$F_1$, \dots,~$F_s$ fit in a $\sigma$-flag of
size~$n$ and a $\sigma$-flag of size~$n$ together with~$F_{s+1}$,
\dots,~$F_t$ fit in~$G$, the identity
\[
p(F_1, \ldots, F_t; G) = \sum_{F \in \Fcal_n^\sigma} p(F_1, \ldots,
F_s; F) p(F, F_{s+1}, \ldots, F_t; G)
\]
holds.
\end{theorem}

Recall from the introduction that
$p(\btwoclaw; G^v) \to p(\bedge; G^v)^2$ as $|G| \to \infty$. The
argument to see this can be rephrased in two steps as follows. First,
since~$G$ is triangle-free,
then~$p(\btwoclaw; G^v) = p(\bedge, \bedge; G^v)$. This can be seen
directly, but is also a consequence of the chain rule. Indeed,
let~$\Hcal = \{\triangleflag\}$ and let~$\bullet$ denote the only type
of size~1. Then $\bullet$-flags $\bedge$, $\bedge$ fit in a
$\bullet$-flag of size~3.
Since~$\Fcal_3^\bullet = \{ \bempty, \bonezero, \boneone, \btwoclaw,
\btwoone \}$, the chain rule gives
\begin{equation}
\label{eq:prod-vague}
p(\bedge, \bedge; G) = \sum_{F' \in \Fcal_3^\bullet} p(\bedge, \bedge;
F') p(F'; G) = p(\btwoclaw; G).
\end{equation}
Second,~$p(\bedge, \bedge; G^v) \to p(\bedge; G^v)^2$ as~$|G| \to
\infty$, that is, density exhibits multiplicative
behavior in the limit:

\begin{theorem}
\label{thm:dens-mult}
If~$F_1$, $F_2$ are fixed $\sigma$-flags, then there exists a
function~$f(n) = O(1/n)$ such that if~$F_1$, $F_2$ fit in a
$\sigma$-flag~$G$, then~$|p(F_1, F_2; G) - p(F_1; G) p(F_2; G)| \leq
f(|G|)$.
\end{theorem}

Identity~\eqref{eq:prod-vague}, that comes from an application of the
chain rule, suggests that there is a relation between the
pair~$(\bedge, \bedge)$ and~$\btwoclaw$. In the next section, we will
use the chain rule to define a product operation on $\sigma$-flags,
and under this product it will hold that~$\bedge \cdot \bedge =
\btwoclaw$. This product will also commute with the density function
in the limit: for $\sigma$-flags~$F_1$ and~$F_2$ we will have~$p(F_1
\cdot F_2; G) \to p(F_1; G) p(F_2; G)$ as~$|G| \to \infty$.


\section{Flag algebras}

In the introduction, we derived the constraint
\begin{align*}
  \phi(\oneedge) + 2 \phi(\twoclaw) &= 3\phi(\edge),
\end{align*}
valid for every $\phi \in \Phi$. If we see~$\phi \in [0,1]^\Gcal$ as a
vector, then this is a linear constraint on the components of~$\phi$.
To enable the use of tools from optimization, mainly duality, we need
to embed our domain into a vector space. We do so by extending $\phi$
linearly to the space~$\R\Gcal$ of formal real linear combinations of
graphs in~$\Gcal$. We could then rewrite the latter constraint as
\[
  \phi(\oneedge+ 2 \twoclaw) =\phi(3\edge),\qquad\text{or even}\qquad
  \phi(\oneedge+ 2 \twoclaw -3\edge) = 0.
\]

One of our main goals is to characterize the linear functionals
on~$\R\Gcal$ that are limit functionals. Instead of describing all the
constraints that characterize limit functionals, it is convenient to
encode some of them algebraically, that is, by modifying the algebraic
structure of~$\R\Gcal$.  The resulting algebraic object will be the
flag algebra, which we construct now for the more general case of
$\sigma$-flags.

Let~$\R\Fcal^\sigma$ be the free vector space over the reals generated
by all $\sigma$-flags, i.e., $\R\Fcal^\sigma$ is the space of all
formal real linear combinations of $\sigma$-flags.
Let~$(A_k)_{k \geq 0}$ be a convergent sequence in $\Fcal^{\sigma}$
and let
\[
\phi(F) = \lim_{k\to\infty} p(F; A_k)
\]
be the pointwise limit of the functions~$p(\cdotwc ; A_k)$.
Extend~$\phi$ linearly to~$\R\Fcal^\sigma$, obtaining a linear
functional. We say that~$\phi$ is the {\it limit functional\/} of the
convergent sequence~$(A_k)_{k \geq 0}$ or, when the sequence itself is
not relevant, that it is a {\it limit functional}.

For any limit functional $\phi$, the chain rule in its
form~\eqref{eq:chain-rule} implies that for every $\sigma$-flag~$F$
and~$n \geq |F|$ we have
\[
\phi(F) = \phi\biggl(\sum_{F' \in \Fcal_n^\sigma} p(F; F') F'\biggr),
\]
that is,
\begin{equation}
F - \sum_{F' \in \Fcal_n^\sigma} p(F; F') F'
\label{eq:kernel-vectors}
\end{equation}
is in the kernel of $\phi$. Instead of enforcing these infinitely many
relations, we might as well just quotient them out. So
let~$\Kcal^\sigma$ be the linear span of vectors of
form~\eqref{eq:kernel-vectors} and
define~$\Acal^\sigma = \R\Fcal^\sigma / \Kcal^\sigma$. This is a
nontrivial vector space, since for every $\sigma$-flag~$F$ we
have~$p(\sigma; F) = 1$, and hence~$\sigma$ is itself not
in~$\Kcal^\sigma$. Since~$\Kcal^\sigma$ is contained in the kernel of
every limit functional, every limit functional is also a linear
functional of~$\Acal^\sigma$.

The main advantage of working with~$\Acal^\sigma$ instead
of~$\R\Fcal^\sigma$ is that it is possible to define a product
on~$\Acal^\sigma$, turning it into an algebra. This product will
conveniently encode the asymptotic multiplicative behavior of
densities described in Theorem~\ref{thm:dens-mult}: for every limit
functional~$\phi$ and~$f$, $g \in \Acal^\sigma$ we will
have~$\phi(f \cdot g) = \phi(f) \phi(g)$.

For $\sigma$-flags~$F$ and~$G$, let~$n$ be any integer
such that~$F$, $G$ fit in a $\sigma$-flag of size~$n$ and set
\begin{equation}
\label{eq:prod-def}
F \cdot G = \biggl(\sum_{H \in \Fcal_n^\sigma} p(F, G; H) H\biggr) +
\Kcal^\sigma.
\end{equation}
This defines a function from~$\Fcal^\sigma \times \Fcal^\sigma$
to~$\Acal^\sigma$ and one may show that the definition is independent
of the choice of~$n$ for each pair~$(F, G)$ of $\sigma$-flags. Now,
extend this function bilinearly
to~$\R\Fcal^\sigma \times \R\Fcal^\sigma$. It is possible to prove
that if~$f \in \Kcal^\sigma$ and~$g \in \R\Fcal^\sigma$,
then~$f \cdot g = \Kcal^\sigma$, whence the bilinear extension is
constant on cosets, and therefore defines a symmetric bilinear form
on~$\Acal^\sigma$, that is, a commutative product.

This turns~$\Acal^\sigma$ into an algebra, the \textit{flag algebra of
  type~$\sigma$}. The product on~$\Acal^\sigma$ is now defined, and we
will use henceforth the natural
correspondence~$f \mapsto f + \Kcal^\sigma$ between~$\R\Fcal^\sigma$
and~$\Acal^\sigma$ without further notice, i.e., we will
omit~$\Kcal^\sigma$ and write~$f$ instead of~$f + \Kcal^\sigma$ for an
element of~$\Acal^\sigma$. Sometimes, namely in~\S\ref{sec:sdp}, it is
important to work with explicit representatives of each coset; in such
cases we will clearly distinguish between cosets and their
representatives.

Under the product just defined for~$\Acal^\sigma$, the type~$\sigma$,
taken as a $\sigma$-flag, is the identity element. The
identity~$\sigma$ can be decomposed in many different ways using
relations~\eqref{eq:kernel-vectors}. Indeed, for
any~$n \geq |\sigma|$, we have
\[
\sigma = \sum_{F \in \Fcal^\sigma_n} p(\sigma; F) F = \sum_{F \in
  \Fcal^\sigma_n} F.
\]

It now follows from Theorem~\ref{thm:dens-mult} that limit functionals
are multiplicative, i.e.,
\[
\phi(f \cdot g) = \phi(f) \cdot \phi(g)
\]
for~$f$, $g \in \Acal^\sigma$.  Since by
construction~$\phi(\sigma) = 1$, every limit functional~$\phi$ is an
algebra homomorphism between~$\Acal^\sigma$ and~$\R$. We denote the
set of all algebra homomorphisms between~$\Acal^\sigma$ and~$\R$
by~$\Hom(\Acal^\sigma, \R)$.

As an example, recall the discussion at the end of the previous
section. When $\Hcal = \{\triangleflag\}$, if we expand the
product $\bedge \cdot \bedge$ as a linear combination of
$\bullet$-flags of size~3, then~$\bedge \cdot \bedge = \btwoclaw$.
Hence every limit functional~$\phi$
satisfies~$\phi(\btwoclaw) = \phi(\bedge \cdot \bedge) =
\phi(\bedge)^2$.

Every limit functional~$\phi$ lies in~$\Hom(\Acal^\sigma, \R)$.
Another obvious constraint that every limit functional $\phi$ must
satisfy is $\phi(F) \geq 0$ for every $\sigma$-flag~$F$, which is not
necessarily true of all homomorphisms.
Call~$\phi \in \Hom(\Acal^\sigma, \R)$ {\it positive\/}
if~$\phi(F) \geq 0$ for every $\sigma$-flag~$F$, and let
$\Hom^+(\Acal^\sigma, \R)$ denote the set of all positive
homomorphisms.

It turns out that these are all the essential properties of a limit
functional. It is clear that every limit functional is a positive
homomorphism. The following theorem of Razborov~\cite{Razborov2007}
establishes the converse, and so positive homomorphisms are precisely
the limit objects of convergent sequences of flags. In particular, the
linear extension of the set~$\Phi$ is precisely
$\Hom^+(\Acal^\nulltype, \R)$.

\begin{theorem}
\label{thm:lovasz-szegedy}
Every limit functional is a positive homomorphism and every positive
homomorphism is a limit functional.
\end{theorem}

Finally, notice that types and flags are defined in terms of the
family~$\Hcal$ of forbidden subgraphs, so this family is encoded in
the construction of the flag algebra~$\Acal^\sigma$ itself.

\section{Downward operator}
\label{sec:downward}

We are really interested in working with $\nulltype$-flags, that is,
unlabeled graphs, so why consider other types altogether? Most times,
in order to obtain results for $\nulltype$-flags, it is necessary to
use other types.  In the introduction, to obtain Mantel's theorem, it
was not enough to work with unlabeled graphs: at some point, we had to
introduce labeled graphs, namely to get~\eqref{eq:intro_label}.

The downward operator maps $\sigma$-flags into $\nulltype$-flags, in
such a way that we can derive valid inequalities for densities of
$\nulltype$-flags from valid inequalities for densities of
$\sigma$-flags. If types can be seen as a form of lifting, then the
downward operator is a projection back to our space of interest.

If~$F$ is a $\sigma$-flag, then~$\forget{F}$ is the $\nulltype$-flag
obtained from~$F$ simply by forgetting the embedding, that is, by
forgetting the vertex labels. For a~$\sigma$-flag~$F$,
let~$q_\sigma(F)$ be the probability that an injective
map~$\theta\colon [k] \to V(F)$ taken uniformly at random is such
that~$(\forget{F}, \theta)$ is a $\sigma$-flag isomorphic to~$F$ and
set
\[
\downop{F}_\sigma = q_\sigma(F) \forget{F},
\]
then extend~$\downop{\cdotwc}_\sigma$ linearly to~$\R\Fcal^\sigma$ to
obtain a linear map from~$\R\Fcal^\sigma$ to~$\R\Fcal^\nulltype$.  One
key property of this map is
that~$\downop{\Kcal^\sigma}_\sigma \subseteq \Kcal^\nulltype$, and
hence~$\downop{\cdotwc}_\sigma$ gives a linear map from~$\Acal^\sigma$
to~$\Acal^\nulltype$, which we call {\it downward operator}.  The main
tool used in the proof of this result is the following lemma, which
relates densities in the labeled and in the unlabeled cases by taking
an average.

\begin{lemma}
\label{lem:downward-exp}
Let~$F$ be a $\sigma$-flag and~$G$ be an $\nulltype$-flag
with~$|G| \geq |F|$ and~$p(\forget{\sigma}; G) > 0$. If~$\theta$ is an
embedding of~$\sigma$ into~$G$ chosen uniformly at random,
then~$p(F; (G, \theta))$ is a random variable and
\[
\expec[p(F; (G, \theta))] = {q_\sigma(F) p(\forget{F}; G)\over
  q_\sigma(\sigma) p(\forget{\sigma}; G)}.
\]
\end{lemma}

Note that equation~\eqref{eq:intro_label} in the introduction follows
trivially from this lemma. Indeed, take~$\sigma = \bullet$ as the type
of size~1 and let $F = \btwoclaw$. Then $\forget{F}=\twoclaw$,
$q_\sigma(F) = 1/3$, $q_\sigma(\sigma)=1$ and
$p(\forget{\sigma}; G)=1$ for any graph $G$. Thus, by
Lemma~\ref{lem:downward-exp},
\begin{equation*}
 \frac{1}{|G|}\sum_{v\in V(G)}p(\btwoclaw; G^v)
=
\expec[p(F; (G, \theta))]
=
    {q_\sigma(F) p(\forget{F}; G)\over
  q_\sigma(\sigma) p(\forget{\sigma}; G)}
=
    \frac{1}{3} p(\twoclaw; G).
\end{equation*}


\section{Conic programming}

For~$f \in \Acal^\sigma$ and a linear functional~$\phi$ in the dual
space $(\Acal^\sigma)^*$ of $\Acal^\sigma$, write~$(\phi, f) = \phi(f)$. The
\textit{semantic cone} of type~$\sigma$ is the set
\[
\Scal^\sigma = \{\, f \in \Acal^\sigma : \text{$(\phi, f) \geq 0$ for
  all~$\phi \in \Hom^+(\Acal^\sigma, \R)$}\,\}.
\]
This is a convex cone and its dual cone
\[
(\Scal^\sigma)^* = \{\, \phi \in (\Acal^\sigma)^* : \text{$(\phi, f)
  \geq 0$ for all~$f \in \Scal^\sigma$}\,\}
\]
contains every nonnegative multiple of functionals
in~$\Hom^+(\Acal^\sigma, \R)$. So, given a graph~$C$,
\begin{equation}
\label{eq:flag-primal}
\max\{\, (\phi, C) : \phi \in \Hom^+(\Acal^\nulltype, \R)\, \}
\leq \max\{\, (\phi, C) : \text{$\phi \in (\Scal^\nulltype)^*$
  and~$(\phi, \nulltype) = 1$}\,\}.
\end{equation}
(Here we may write ``max'' instead of ``sup'' because
$\Hom^+(\Acal^\nulltype, \R)$ is compact. Actually, equality holds by
the bipolar theorem.)

The optimization problem on the right-hand side above is a
\textit{conic programming problem}. It asks us to maximize a linear
function~$\phi \mapsto (\phi, C)$ over the intersection of a cone,
namely~$(\Scal^\nulltype)^*$, and an affine subspace, in our case
determined by the linear equation~$(\phi, \nulltype) = 1$.

This conic programming problem has a \textit{dual problem}, namely
\begin{equation}
\label{eq:flag-dual}
\min\{\, \lambda : \text{$\lambda\nulltype - C \in \Scal^\nulltype$
  and~$\lambda \in \R$}\,\},
\end{equation}
where the optimization variable is~$\lambda$. (We may write ``min''
instead of ``inf'' because the feasible region is a closed half-line
in~$\R$.)

Weak duality holds: any feasible solution of the dual has larger or
equal objective value than any feasible solution of the
primal. Indeed, if~$\phi \in (\Scal^\nulltype)^*$ is such
that~$(\phi, \nulltype) = 1$ and~$\lambda \in \R$ is such
that~$\lambda\nulltype - C \in \Scal^\nulltype$, then
\[
0 \leq (\phi, \lambda\nulltype - C) = \lambda - (\phi, C).
\]
Actually, it is easy to show that there is no duality gap, that is,
that primal and dual have the same optimal value. Even more: the
problem on the left-hand side of~\eqref{eq:flag-primal} has the same
optimal value of the dual problem~\eqref{eq:flag-dual}, and so all
three optimization problems in~\eqref{eq:flag-primal}
and~\eqref{eq:flag-dual} have the same optimal value.  Indeed, notice
that the maximum on the left-hand side of~\eqref{eq:flag-primal} is
equal to
\[
\min\{\, \lambda : \text{$(\phi, C) \leq \lambda$ for
  all~$\phi \in \Hom^+(\Acal^\nulltype, \R)$}\,\}.
\]
Now, $\lambda \geq (\phi, C)$ for all~$\phi \in
\Hom^+(\Acal^\nulltype, \R)$ if and only if~$(\phi, \lambda \nulltype
- C) \geq 0$ for all~$\phi \in \Hom^+(\Acal^\nulltype, \R)$ if and
only if~$\lambda\nulltype - C \in \Scal^\nulltype$, as we wanted.

To find an upper bound for~$\exparam(C, \Hcal)$ we work with the dual
problem~\eqref{eq:flag-dual}. One advantage is that we do not need to
solve this problem to optimality to find an upper bound, since any
feasible solution provides an upper
bound. Solving~\eqref{eq:flag-dual} to optimality is the same as
solving the primal problem to optimality, which is the same as
computing~$\exparam(C, \Hcal)$. 

One way to simplify the dual problem~\eqref{eq:flag-dual} is to
replace~$\Scal^\nulltype$ with a
cone~$\Ccal \subseteq \Scal^\nulltype$ for which it is easier to solve
the resulting problem. Obviously, we still get a valid upper bound. We
seem to have taken a tortuous path since the introduction, where we
stated our goal of finding a relaxation of~$\Phi$, of
which~$\Hom^+(\Acal^\nulltype, \R)$ is the linear extension, but that
is exactly what we achieved, albeit via the dual:
\[
\Hom^+(\Acal^\nulltype, \R) \subseteq \{\, \phi \in
(\Acal^\nulltype)^* : \text{$(\phi, f) \geq 0$ for all~$f \in \Ccal$
  and~$(\phi, \nulltype) = 1$}\,\}.
\]

What are some~$f \in \Acal^\sigma$ that belong to the semantic
cone~$\Scal^\sigma$?  Since a positive homomorphism~$\phi$ is by
definition nonnegative on every $\sigma$-flag~$F$, then any conic
combination of $\sigma$-flags is in the semantic cone. Another class
of vectors in the semantic cone is the class of vectors that are sums
of squares. We say that~$f \in \Acal^\sigma$ is a \textit{sum of
  squares} if there are~$g_1$, \dots,~$g_t \in \Acal^\sigma$ such
that~$f = g_1^2 + \cdots + g_t^2$. Then for any positive
homomorphism~$\phi$ (actually, for any homomorphism) we
have~$\phi(f) = \phi(g_1)^2 + \cdots + \phi(g_t)^2 \geq 0$. The class
of sum-of-squares vectors is particularly interesting because it is
computationally tractable, as we will soon see.  Finally, the
downward operator maps the semantic cone~$\Scal^\sigma$ of
type~$\sigma$ into the semantic cone~$\Scal^\nulltype$ of
type~$\nulltype$:

\begin{theorem}
\label{thm:downop}
The image of~$\Scal^\sigma$ under~$\downop{\cdotwc}_\sigma$ is a
subset of~$\Scal^\nulltype$.
\end{theorem}

\noindent
This gives yet another way to obtain vectors in~$\Scal^\nulltype$, by
first considering a type~$\sigma$, then obtaining a vector
in~$\Acal^\sigma$ (a sum-of-squares vector, for instance), and then
using the downward operator.


\section{The semidefinite programming method}
\label{sec:sdp}

Semidefinite programming is conic programming over the cone of
positive semidefinite matrices. Using sum-of-squares vectors
in~$\Acal^\sigma$ and the downward operator, we may define a family of
tractable cones contained in~$\Scal^\nulltype$. Then using
semidefinite programming it is possible to write down optimization
problems that provide upper bounds to~\eqref{eq:flag-dual}. This
approach is known as the \textit{semidefinite programming method}. Its
main advantages are that writing down the semidefinite programming
problems is mostly a mechanical affair, that can even be automated
(and has been; see for instance flagmatic~\cite{RavryV2013}), and
solving the resulting problems can be done with a computer.

There is a well-known relation between sums-of-squares polynomials and
positive semidefinite matrices (see e.g.\ the exposition by
Laurent~\cite{Laurent2008}). We now establish the analogous relation
between sums-of-squares vectors in~$\Acal^\sigma$ and positive
semidefinite matrices. The \textit{degree} of a
vector~$f \in \R\Fcal^\sigma$ is the largest size of a flag appearing
with a nonzero coefficient in the expansion of~$f$; by convention, the
degree of~$0$ is~$-1$. The notion of degree can be extended
to~$\Acal^\sigma$, by setting the degree
of~$f + \Kcal^\sigma \in \Acal^\sigma$ to be the smallest degree of
any~$g \in f + \Kcal^\sigma$. For a type~$\sigma$
and~$n \geq |\sigma|$,
let~$v_{\sigma,n}\colon \Fcal^\sigma_n \to \Acal^\sigma$ be the
canonical embedding, i.e., $v_{\sigma, n}(F) = F$ for
all~$F \in \Fcal_n^\sigma$.

\begin{theorem}
\label{thm:flag-sos}
If~$f \in \Acal^\sigma$ and $n \geq |\sigma|$, then there are
vectors~$g_1$, \dots,~$g_t \in \Acal^\sigma$ for some \(t \geq 1\), each of
degree at most~$n$, such that~$f = g_1^2 + \cdots + g_t^2$ if and only
if there is a positive semidefinite matrix~$Q\colon \Fcal_n^\sigma
\times \Fcal_n^\sigma \to \R$ such that~$f = v_{\sigma,n}^\tp Q
v_{\sigma,n}$.
\end{theorem}

\begin{proof}
Suppose that there are vectors~$g_1$, \dots,~$g_t$ as
described. Modulo~$\Kcal^\sigma$, every $\sigma$-flag of size~$m$ can
be written as a linear combination of $\sigma$-flags of any fixed size
greater than~$m$. So by hypothesis we can take from each coset~$g_i +
\Kcal^\sigma$ a representative $\hat{g}_i \in \R\Fcal^\sigma$ which is
a linear combination of $\sigma$-flags of size~$n$.

Let~$c_i$ be the vector of coefficients of~$\hat{g}_i$, in such a
way that~$\hat{g}_i = c_i^\tp v_{\sigma,n}$. Then
\[
\hat{g}_1^2 + \cdots + \hat{g}_t^2 = \sum_{i=1}^t (c_i^\tp v_{\sigma,n})^2 = \sum_{i=1}^t
v_{\sigma,n}^\tp c_i c_i^\tp v_{\sigma,n},
\]
and we may take~$Q = c_1 c_1^\tp + \cdots + c_t c_t^\tp$.

For the converse, say there is a positive semidefinite matrix~$Q$ as
described. Then for some~$t$ there are vectors~$c_1$, \dots,~$c_t$
such that~$Q = c_1 c_1^\tp + \cdots + c_t c_t^\tp$.  But then~$g_i =
c_i^\tp v_{\sigma,n}$ has degree at most~$n$
in~$\Acal^\sigma$. Moreover,~$f = g_1^2 + \cdots + g_t^2$, as we
wanted.
\end{proof}

Let us describe the semidefinite programming method by applying it to
Mantel's theorem. Fix~$\Hcal = \{\triangleflag\}$. We have the
following $\nulltype$-flags of sizes~2 and~3:~$\nonedge$, $\edge$,
$\emptythree$, $\oneedge$, and~$\twoclaw$. There is also only one type
of size~1, namely the graph on one vertex, which we denote
by~$\bullet$. These are the $\bullet$-flags of sizes~2
and~3:~$\bnonedge$, $\bedge$, $\bempty$, $\bonezero$, $\boneone$,
$\btwoclaw$, and~$\btwoone$.

Write~$v = v_{\bullet,2}$, so that in vector notation we
have~$v = (\bnonedge, \bedge)$. From Theorem~\ref{thm:flag-sos},
if~$Q\colon\Fcal^\bullet_2 \times \Fcal^\bullet_2 \to \R$ is a
positive semidefinite matrix, then~$v^\tp Q v$ belongs to the semantic
cone~$\Scal^\bullet$ of type~$\bullet$, and hence from
Theorem~\ref{thm:downop} we have that~$\downop{v^\tp Q v}_\bullet$
belongs to the semantic cone~$\Scal^\nulltype$ of
type~$\nulltype$. Since any conic combination~$r$ of $\nulltype$-flags
belongs to the semantic cone~$\Scal^\nulltype$, we have that
\[
r + \downop{v^\tp Q v}_\bullet \in \Scal^\nulltype
\]
for every conic combination~$r$ of $\nulltype$-flags and every positive
semidefinite matrix~$Q$.

So, recalling~\eqref{eq:flag-dual}, any feasible solution of the
following optimization problem gives an upper bound
to~$\exparam(\edge, \{\triangleflag\})$:
\begin{equation}
\label{eq:mantel-first-opt}
\begin{array}{rl}
\min&\lambda\\
&\lambda\nulltype - \edge = r + \downop{v^\tp Q v}_\bullet,\\
&\text{$r$ is a conic combination of $\nulltype$-flags},\\
&\text{$Q\colon \Fcal^\bullet_2 \times \Fcal^\bullet_2 \to \R$ is
  positive semidefinite.}
\end{array}
\end{equation}
This problem is not quite a semidefinite programming problem: the
first identity above is an identity between vectors
in~$\Acal^\nulltype$, not a linear constraint on~$\lambda$ and the
entries of~$Q$. This identity can be translated, however, into several
linear constraints, as follows.

If~$A$ and~$B$ are~$n \times n$ matrices,
write~$\langle A, B\rangle = \tr A^\tp B = \sum_{i,j=1}^n A_{ij}
B_{ij}$. Then
\[
\downop{v^\tp Q v}_\bullet = \downop{\langle v v^\tp, Q
  \rangle}_\bullet = \langle \downop{v v^\tp}_\bullet, Q \rangle.
\]
Here, notice that~$v v^\tp$ is a matrix. The downward operator, when
applied to the matrix~$v v^\tp$, is applied entrywise and yields a
matrix of the same dimensions as the result.

So the first constraint in~\eqref{eq:mantel-first-opt} can be rewritten
as
\begin{equation}
\label{eq:flag-identity}
\lambda\nulltype - \edge = r + \langle \downop{v v^\tp}_\bullet, Q
\rangle,
\end{equation}
which is still an identity between elements of~$\Acal^\nulltype$. To
test the above identity, we may choose a large enough~$N$ and use the
chain rule to expand both left and right-hand sides as linear
combinations of $\nulltype$-flags of size~$N$. If the coefficients
coincide, then equality holds. This is only a sufficient condition
however: for a fixed~$N$, equality may hold in~$\Acal^\nulltype$ even
though the coefficients differ, but it is not hard to show that there
is always some~$N$ for which equality holds if and only if the
coefficients coincide.

To make things precise, we have to choose for~$\edge$, $r$, and every
element of~$\Acal^\nulltype$ in~$v v^\tp$ a representative
in~$\R\Fcal^\nulltype$. As a representative
of~$\edge \in \Acal^\nulltype$ we may
choose $\edge \in \R\Fcal^\nulltype$. For~$v v^\tp$ proceed as
follows: use the definition of product in~$\Acal^\bullet$ to get
\[
v v^\tp = \begin{pmatrix}
\bempty + \bonezero&\frac{1}{2}(\boneone + \btwoone)\\[5pt]
\frac{1}{2}(\boneone + \btwoone)&\btwoclaw
\end{pmatrix}
\]
and then apply the downward operator to get
\[
\downop{v v^\tp}_\bullet = \begin{pmatrix}
\emptythree + \frac{1}{3}\oneedge&\frac{1}{3}(\oneedge + \twoclaw)\\[5pt]
\frac{1}{3}(\oneedge + \twoclaw)&\frac{1}{3}\twoclaw
\end{pmatrix}.
\]
We will deal with~$r$ below in a different way (actually, we will get
rid of it).  Notice we could have chosen different
representatives. For instance, we could have expanded the products
in~$v v^\tp$ using $\bullet$-flags of size~6, say. All that matters,
however, is to choose representatives, and it is usually a good idea to
choose representatives of smallest possible degree.

Now we are working exclusively with representatives
in~$\R\Fcal^\nulltype$. For a given~$N > 0$ and
fixed~$G \in \Fcal_N^\nulltype$, extend~$F \mapsto p(F; G)$ linearly
to~$\Fcal_N^\nulltype$. If for every~$G \in \Fcal^\nulltype_N$ we have
\begin{equation}
\label{eq:linear-constraints}
p(\lambda\nulltype - \edge; G) = p(r; G) + p(\langle \downop{v
  v^\tp}_\bullet, Q\rangle; G),
\end{equation}
then~\eqref{eq:flag-identity} holds. Conversely,
if~\eqref{eq:flag-identity} holds, then for some~$N > 0$
\eqref{eq:linear-constraints} holds for
every~$G \in \Fcal_N^\nulltype$ (this requires a short argument though).

Now, $p(r; G)$ is the coefficient of~$G$ in~$r$; then, since~$r$ is a
conic combination, $p(r; G) \geq 0$ for
every~$G \in \Fcal_N^\nulltype$. Together with linearity this implies
that we may rewrite~\eqref{eq:linear-constraints} equivalently as
\begin{equation}
\label{eq:linear-constraints-rewrite}
\lambda - p(\edge; G) \geq \langle p(\downop{v v^\tp}_\bullet; G),
Q\rangle,
\end{equation}
where~$p(\,\cdot\,; G)$ is applied entrywise to~$v v^\tp$. Notice
that~$p(\edge; G)$ is a number and~$p(\downop{v v^\tp}_\bullet; G)$ is
a matrix of numbers, so for each~$G \in \Fcal_N^\nulltype$ the above
inequality is a linear constraint on~$\lambda$ and the entries of~$Q$.

In our case, we may take~$N = 3$.
Then~\eqref{eq:linear-constraints-rewrite} gives rise to one linear
constraint for each of the $\nulltype$-flags of size~3:
\[
\begin{array}{cl}
\hbox{$\nulltype$-flag}&\multicolumn{1}{c}{\hbox{constraint}}\\[3pt]
\emptythree&\lambda \geq \langle \bigl(\begin{smallmatrix}
1&0\\
0&0
\end{smallmatrix}\bigr), Q\rangle,\\[3pt]
\oneedge&\lambda - 1/3 \geq \langle \bigl(\begin{smallmatrix}
1/3&1/3\\
1/3&0
\end{smallmatrix}\bigr), Q\rangle,\\[3pt]
\twoclaw&\lambda - 2/3 \geq \langle \bigl(\begin{smallmatrix}
0&1/3\\
1/3&1/3
\end{smallmatrix}\bigr), Q\rangle.
\end{array}
\]

In this way we may rewrite problem~\eqref{eq:mantel-first-opt},
obtaining a semidefinite programming problem that gives an upper bound
to the optimal value of~\eqref{eq:mantel-first-opt}, and hence also
to~$\exparam(\edge, \{\triangleflag\})$. This problem is not
necessarily equivalent to~\eqref{eq:mantel-first-opt}, since for a
given~$N$ equality in the algebra may hold even though the linear
constraints are not satisfied.

Now, it is easy to check that~$\lambda = 1/2$
and~$Q =
\frac{1}{2}\bigl(\begin{smallmatrix}1&-1\\-1&1\end{smallmatrix}\bigr)$
form a feasible solution of this semidefinite programming problem (and
hence also of~\eqref{eq:mantel-first-opt}), and so we have Mantel's
theorem.

All the steps of the semidefinite programming method are contained in
the example we worked out above. In general, however, one may choose a
finite set~$\Tcal$ of types instead of only one type and consider the
vectors in~$\Scal^\nulltype$ given by
\[
r + \sum_{\sigma \in \Tcal} \downop{v_{\sigma,n_\sigma}^\tp Q_\sigma
  v_{\sigma,n_\sigma}}_\sigma,
\]
where~$r$ is a conic combination of $\nulltype$-flags,
$n_\sigma \geq |\sigma|$, and each~$Q_\sigma$ is a positive semidefinite
matrix. Choosing more types makes the problem larger, but also
potentially stronger.


\section{Summary}

The theory of flag algebras provides a powerful, unifying approach for
extremal problems involving a host of combinatorial structures. Its
novelty is that it allows the formulation of relaxations for such
problems using conic programming, which can be further relaxed to
semidefinite programming problems, thus enabling the use of a computer
to obtain bounds. Most importantly, the computed bounds are often
tight. Hence, the theory yields relaxations that achieve the desired
trade-off of computational tractability and high-quality bounds.

We have only scratched the surface of the theory of flag algebras.
Many optimization aspects of the semidefinite method, such as the use
of complementary slackness to obtain further constraints on the
optimal solutions for~\eqref{eq:flag-primal}, were left
out. Complementary slackness can be useful to show properties of all
increasing sequences~$(G_k)_{k \geq 0}$ that attain
$\exparam(C, \Hcal)$, an important issue in extremal combinatorics.
Razborov~\cite{Razborov2007} further developed other methods involving
flag algebras, such as the differential method and the inductive
method.

Techniques involving flag algebras have been used to obtain many
significant new results such as: computing the minimal number of
triangles in graphs with given density in~\cite{Razborov2008,PikhurkoR2016}, computing the
maximum number of pentagons in triangle-free graphs
in~\cite{Grzesik2012,HatamiHKN2013}, and obtaining new advances towards the
Cacceta-Häggkvist conjecture in~\cite{Razborov2013}. Besides being
applied in the context of graphs and digraphs, flag algebras have also
been successfully used in the setting of colored graphs
(e.g.~\cite{BaberT2014,CummingsKPSTY2013}) and of permutations
(e.g.~\cite{BaloghHLPUV2014}). For many more references, see the
thesis of Grzesik~\cite{Grzesik2014}.

\end{document}